\documentclass[runningheads]{llncs}
\usepackage{graphicx, amssymb, amsmath}
\newcommand{\Z}{\mathbb{Z}}
\newcommand{\N}{\mathbb{N}}

\newcommand{\WP}{\mathrm{WP}}
\newcommand{\Sym}{\mathrm{Sym}}
\newcommand{\lang}{\mathcal{L}}

\newcommand{\dom}{\mathrm{dom}}

\newcommand{\AU}{\mathcal{S}}
\newcommand{\TT}{\mathcal{T}}
\newcommand{\PP}{\mathcal{P}}

\begin{document}
\title{Four heads are better than three\thanks{Research supported by Academy of Finland grant 2608073211.}}
%
%
\author{Ville Salo\inst{1}\orcidID{0000-0002-2059-194X}}
\authorrunning{V. Salo}
%
\institute{University of Turku}
\maketitle              
\begin{abstract}
We construct recursively-presented finitely-generated torsion groups which have bounded torsion and whose word problem is conjunctive equivalent (in particular positive and Turing equivalent) to a given recursively enumerable set. These groups can be interpreted as groups of finite state machines or as subgroups of topological full groups, on effective subshifts over other torsion groups. We define a recursion-theoretic property of a set of natural numbers, called impredictability. It roughly states that a Turing machine can enumerate numbers such that every Turing machine occasionally incorrectly guesses (by either halting or not) whether they are in the set, even given an oracle for a prefix of the set. We prove that impredictable recursively enumerable sets exist. Combining these constructions and slightly adapting a result of [Salo and T\"orm\"a, 2017], we obtain that four-headed group-walking finite-state automata can define strictly more subshifts than three-headed automata on a group containing a copy of the integers, confirming a conjecture of [Salo and T\"orm\"a, 2017]. These are the first examples of groups where four heads are better than three, and they show the maximal height of a finite head hierarchy is indeed four.

\keywords{torsion groups \and finite-state automata \and subshifts \and recursion theory.}
\end{abstract}

\section{Introduction}

In this paper, we combine construction techniques from group theory and recursion theory to verify a conjecture  the author and T\"orm\"a made in \cite{SaTo17} about group-walking automata. The a priori motivation was this conjecture, but the constructions and definitions we give may be of independent interest.

\subsection{The results}

\begin{theorem}
\label{thm:IntroConstruction}
Let $A$ be $\Sigma^0_1$. Then there exists a recursively presented torsion group with bounded torsion, whose word problem is Turing equivalent to $A$.
\end{theorem}

The exact types of reduction are that $A$ many-one reduces to the word problem, and the word problem conjunctively reduces to $A$.

Our recursion-theoretic contributions are of a somewhat technical nature (though not particularly difficult). We define a possibly new notion called $\phi$-impredictability (Definition~\ref{def:Impredictable}), which roughly states that a Turing machine can enumerate numbers such that every Turing machine occasionally incorrectly guesses (by either halting or not) whether they are in the set, even given an oracle for a prefix of the set. 

\begin{theorem}
\label{thm:ExistsApredictableIntro}
For every total recursive function $\phi$, there exists a $\phi$-impredictable $\Sigma^0_1$ set.
\end{theorem}

We now state our main ``navigation-theoretic'' contribution. In \cite{SaTo17}, for a f.g.\ (finitely-generated) group $G$ and $n \in N$, the class of $G$-subshifts defined by group-walking finite-state machines with $n$ heads is denoted by $\mathcal{S}(G, n)$. See \cite{SaTo17} or Section~\ref{sec:SubshiftCorollary} for the definitions. The following is a slight adaptation of a result of \cite{SaTo17}.

\begin{lemma}
\label{lem:explist}
If $G$ has bounded torsion and has $\phi$-impredictable word problem for fast-enough growing $\phi$, then $\mathcal{S}(G \times \Z, 3) \subsetneq \mathcal{S}(G \times \Z, 4)$. 
\end{lemma}

For example $\phi = \exp \circ \exp \circ \exp \circ \exp \circ \exp$ is fast enough. By putting the above results together, we obtain that four heads are better than three, as claimed in the title.

\begin{theorem}
\label{thm:FourBetterThanThree}
There exists a finitely-generated recursively-presented group $G$ containing a copy of the integers, such that $\mathcal{S}(G, 3) \subsetneq \mathcal{S}(G, 4)$.
\end{theorem}

See Section~\ref{sec:Context} for some context for this result, and a new conjecture.

\subsection{Some relevant existing work}

The main group-construction result (recursively presented groups with bounded torsion and with a word problem of a prescribed difficulty) uses the idea from \cite{SaSc16a,BaKaSa16} of groups of finite-state machines. It also uses existing torsion groups as a black-box, in particular it is an application of the deep theory that arose from the Burnside problem \cite{Ad11}.

The idea of hiding information into the word problem is of course not a new idea in combinatorial group theory, but we are not aware of it appearing previously in the context of torsion groups. The Dehn monsters from \cite{MyOs11}, recursively presented groups where no infinite set of distinct elements can be enumerated, seem strongly related. (Our construction cannot be used to produce such groups due to using another group as a black box, but it is possible that their construction can be adapted to produce our result.)

Topological full groups (already on $\Z$) are a well-known source of interesting examples of groups \cite{Ma06,GrMe14,JuMo12}, and our groups can also be interpreted as subgroups of topological full groups of subshifts on torsion groups. In a symbolic dynamics context, \cite{GuJeKaVa18} (independently) uses a similar construction to prove that the automorphism groups of multidimensional SFTs can have undecidable word problem.

Dan Turetsky showed in the MathOverflow answer \cite{MOTu20} that the halting problem is not $\phi$-impredictable for some total recursive $\phi$, and that there exists a $\Sigma^0_1$ set which is simultaneously $\phi$-impredictable for all total recursive $\phi$. The latter result of course implies Theorem~\ref{thm:ExistsApredictableIntro}.

\subsection{Head hierarchies}
\label{sec:Context}

For context, we state what is now known about the head hierarchies, and state a bold conjecture. See Definition~\ref{def:Automata} for the definitions.

\begin{theorem}
\label{thm:WhatIsKnown}
For an infinite finitely-generated group $G$, denote by $h(G) \in \N_{> 0} \cup \{\infty\}$ the supremum of $n$ such that $\mathcal{S}(G, n-1) \subsetneq \mathcal{S}(G, n)$. Then
\[ \{3,4,\infty\} \subset \mathcal{H} = \{h(G) \;|\; G \mbox{ an infinite f.g.\ group}\} \subset \{2,3,4,\infty\}. \]
\end{theorem}

Our contribution is the $4$ on the left, which is notable because the maximal finite height is now known.

\begin{proof}[Proof sketch]
The fact $4 \in \mathcal{H}$ follows from Theorem~\ref{thm:FourBetterThanThree}. The fact $h \in \mathcal{H} \cap \N \implies h \leq 4$ is proved in \cite{SaTo17}, and the facts $3 \in \mathcal{H}$ and $\infty \in \mathcal{H}$ are proved in \cite{SaTo17} and \cite{SaTo15} respectively. The fact $1 \notin \mathcal{H}$ is seen as follows: In  \cite{SaTo15} it is shown that $\AU(\Z, 1) \subsetneq \AU(\Z, 2)$ and the proof easily generalizes to virtually $\Z$ groups. On the other hand it is easy to see that the sunny-side-up $X_{\leq 1}$ separates one from four heads, i.e.\ $X_{\leq 1} \in \AU(G, 4) \setminus \AU(G, 1)$, for all f.g.\ infinite groups $G$ which are not virtually $\Z$, which can be proved using methods of \cite{SaTo17} and \cite{SaTo15}.
\end{proof}

We conjecture that four heads are needed if and only if the word problem is undecidable:

\begin{conjecture}
\label{con:Char}
Let $G$ be a finitely-generated infinite group which is not torsion. Then
\begin{itemize}
\item if $\Z \leq G$ and $G$ has decidable word problem, then $h(G) = 3$,
\item if $\Z \leq G$ and $G$ has undecidable word problem, then $h(G) = 4$.
\end{itemize}
\end{conjecture}

The upper bounds are known, and $h(G) = \infty$ is known for torsion groups \cite{SaTo17}. The first item is the join of conjectures in \cite{SaTo17} and \cite{SaTo15}. It is known that the sunny-side-up subshift $X_{\leq 1}$ does not prove $h(\Z) = h(\Z^2) = 3$, as on these groups two heads suffice for it. We do not know any other non-torsion groups where the sunny-side-up requires fewer heads than the above conjecture would suggest.

Settling this conjecture in the positive would not be the end of the story. A more refined invariant than $h(G)$ would be to ask what the precise set of $n$ such that $\mathcal{S}(G \times \Z, n-1) \subsetneq \mathcal{S}(G \times \Z, n)$ is. In particular this is of interest when $G$ is a torsion group; \cite{SaTo17} only gives an affine function $f(n) \sim 3n$ such that $\mathcal{S}(G \times \Z, n) \subsetneq \mathcal{S}(G \times \Z, f(n))$.

These results are of course about just one way to associate subshift classes to $n$-headed automata. We believe the results are relatively robust to changes in the definition, but some details are critical. It is in particular open what happens when $G$ is a torsion group and the heads are allowed to communicate over distances, i.e.\ if they have a shared state.

\section{Preliminaries}

For two functions $f, g$ write $f = O(g)$ (resp. $f = \Omega(g)$) if for some choice of $a > 0$, $f(n) \leq ag(n)$ (resp. $f(n) \geq ag(n)$) for large enough $n$. Write $f = \Theta(g)$ if $f = O(g)$ and $f = \Omega(g)$. Write $f \sim g$ for $f(n)/g(n) \rightarrow 1$. For $a,b \in \N$ write $a \;|\; b$ for ``$a$ divides $b$''.

We assume some familiarity with computability/recursion theory, but we state some (not necessarily standard) conventions. We identify partial computable functions with Turing machines, and also their G\"odel numbers. Let $\TT \subset \mathbb{N}$ be the set of codings of Turing machines which halt on every input, and write $\PP = \mathbb{N}$ for all partial recursive functions. ``Recursive'' means the same as ``computable'' and refers to the existence of a Turing machine, which always halts unless the function is explicitly stated to be partial. If $\chi$ is a Turing machine, we write $\chi(p){\downarrow}$ if $\chi$ halts on input $p$. A partial (not necessarily computable) function from $A$ to $B$ is denoted $g : A \nrightarrow B$.

Let us recall some basic definitions of reductions. A set $A$ \emph{many-one reduces} to a set $B$ if there is $\phi \in \TT$ such that $n \in A \iff \phi(n) \in B$. For the following three reductions, we give quantitative versions, so we have a handle on the rate at which reduction happens. The definitions are stated in terms of characteristic sequences, but correspond to their usual meaning.

We say $B$ \emph{weakly truth table reduces}, or \emph{wtt-reduces}, to $A$ if there exists $g \in \PP$ and nondecreasing $\beta \in \TT$ with $\beta(n) \rightarrow \infty$ such that when applied to words as $g : \{0,1\}^* \nrightarrow \{0,1\}^*$, if $x \in \{0,1\}^\N$ is the characteristic function of $A$ and $y$ that of $B$, we have that $g(w)$ is defined on prefixes of $y$, and
\[ w \prec y \implies (g(w) \prec x \wedge |g(w)| \geq \beta(|w|)). \]
We call $\beta$ the \emph{rate} of the reduction.


We say $B$ \emph{positively reduces} to $A$ if it wtt-reduces to $A$, for some $\beta \in \TT$, with the following additional properties for $g$: $g \in \TT$, $|g(w)| = \beta(|w|)$ for all $w$, and $g$ is \emph{monotone} in the sense that $u \leq v \implies g(u) \leq g(v)$, where $\leq$ is letterwise comparison. If further $g(u)_i$ only depends on whether $u \geq w_{i,|u|}$ for some $w_{i,|u|}$ computable from $i, |u|$, then $B$ \emph{conjunctively reduces} to $A$.


We say $B$ \emph{Turing reduces} to $A$ if there is an $A$-oracle machine that can determine membership in $B$. Clearly many-one reducibility implies conjunctive reducibility implies positive reducibility implies weak truth table reducibility implies Turing reducibility.

We assume some familiarity with group theory, but state some conventions. Our groups are discrete and mostly finitely-generated. Finitely-generated groups come with a finite generating set, which we usually do not mention. The identity of a group $G$ is denoted by $e_G$ (or just $e$). For words $u, v \in S^*$ where $S$ is the generating set, write $u \approx_G v$ (or just $u \approx v$ if $G$ is clear from context) when $u$ and $v$ represent (evaluate to) the same element of $G$. For $g, h \in G$, write $[g, h] = g^{-1}h^{-1}gh$ for the \emph{commutator} of $g$ and $h$.

The \emph{word problem} of a group $G$ is the following subset $W \subset \N$: Let $S$ be the fixed symmetric generating set, and order elements of $S^*$ (finite words over $S$) first by length and then lexicographically. Include $n \in W$ if the $n$th word evaluates to the identity of $G$. Slight inconvenience is caused by linearizing the word problem this way, but on the other hand sticking to subsets of $\N$ slightly simplifies the discussion in Section~\ref{sec:Impredictability}. We denote the word problem of $G$ as $\WP(G) \subset \N$.

If a countable group $G$ acts on a compact zero-dimensional space $X$, the corresponding \emph{topological full group} is the smallest group of homeomorphisms $g : X \to X$ which contains every homeomorphism $g$ with the following property: there exists a  clopen partition $(P_i)_{i = 1}^k$ of $X$ and $g_i \in G$ such that for $\forall i \in \{1,..,k\}: \forall x \in P_i: gx = g_ix$. It turns out that the group contains precisely such homeomorphisms, i.e.\ they are closed under composition. One may think of $((P_i)_i, (g_i)_i)$ as a local rule for $g$.

A group is \emph{torsion} if all elements have finite order, i.e.\ $G$ does not contain a copy of the integers. The \emph{torsion function} of a torsion group is $T_G : \N \to \N$ defined by
\[ T_G(n) = \sup \{\mathrm{ord}(g) \;|\; g \in G, |g| \leq n\} \]
where $\mathrm{ord}(g) = |\langle g \rangle|$ and $|g|$ is the word norm with respect to the implicit generating set. A group is of \emph{bounded torsion} if $T_G(n) = O(1)$.

In \cite{Iv94} Ivanov shows that the \emph{free Burnside group}
\[ B(2,2^{48}) = \langle a, b \;|\; \forall w \in \{a,b,a^{-1},b^{-1}\}^*: w^{2^{48}} \approx e \rangle \]
is infinite and has decidable word problem (\cite[Theorem~A]{Iv94}), and we obtain

\begin{lemma}
There exists a finitely-generated torsion group with bounded torsion and decidable word problem.
\end{lemma}

The decision algorithm is given in \cite[Lemma~21.1]{Iv94}. For a survey on the Burnside problem see \cite{Ad11}. Theorem~\ref{thm:FourBetterThanThree} could be proved using any torsion group with recursive torsion function (e.g.\ the Grigorchuk group), with minor modifications, as explained after the proof of Theorem~\ref{thm:FourBetterThanThree}. Our technical construction results (in particular Theorem~\ref{thm:Construction}) are stated and proved without assuming the existence of infinite f.g.\ groups of bounded torsion, though obviously to obtain Theorem~\ref{thm:IntroConstruction} the existence of one is necessary (since it states the existence of one).

We assume some familiarity with symbolic dynamics on groups (see \cite{AuBaSa17} for more information), but state some conventions. If $G$ is a group, and $\Sigma$ a finite set, $\Sigma^G$ with the product topology and the $G$-action $gx_h = x_{g^{-1}h}$ is called the \emph{full shift}, and the actions $(g, x) \mapsto gx$ are called \emph{shifts}. A \emph{subshift} is a topologically closed subset $X$ satisfying $GX = X$. A particularly important subshift is the \emph{sunny-side-up} subshift $X_{\leq 1} = \{x \in \{0,1\}^G \;|\; \sum_{g \in G} x_g \leq 1\}$. Equivalently a subshift is defined by a family of \emph{forbidden patterns}, i.e.\ a (possibly infinite) family of clopen sets that the \emph{orbit} $Gx$ of $x$ may not intersect. If $G$ has decidable word problem then $X \subset \Sigma^G$ is \emph{effective} if there exists a Turing machine that enumerates a family of forbidden patterns which defines the subshift. A \emph{cellular automaton} on a subshift is a continuous shift-commuting self-map of it.

\section{Impredictability}
\label{sec:Impredictability}

\begin{definition}
\label{def:Impredictable}
For a function $\phi$, a set $A \subset \mathbb{N}$ is \emph{$\phi$-impredictable} if
\[ \exists \psi \in \TT: \forall \chi \in \PP: \exists^\infty p: \psi(p) \in A \iff \chi(p, A \upharpoonright \phi(p))\!\downarrow. \]
\end{definition}

To unravel this definition a bit, we want to have a Turing machine $\psi$ which always halts and gives us positions $\psi(p)$ on the number line $\N$ so that, for any Turing machine $\chi$, $\psi(p) \in \N$ is just the halting information about $\chi$ on input $p$, even if $\chi$ is allowed oracle access to the first $\phi(p)$ bits of $A$. Since $\chi$ is quantified universally, this definition is one way to formalize the idea that it is hard to predict whether $\psi(p) \in A$ even given access to the first $\phi(p)$ bits of $A$.

\begin{lemma}
\label{lem:Faster}
Let $\phi', \phi \in \TT$ and $\phi'(n) \geq \phi(n)$ for all large enough $n$. Then every $\phi'$-impredictable set is $\phi$-impredictable, with the same choice of $\psi$.
\end{lemma}

\begin{proof}
Let $\phi, \phi' \in \TT$ and suppose $A$ is $\phi'$-impredictable. We need to show $\phi$-impredictability where $\phi(n) \leq \phi'(n)$ for all large enough $n$. Clearly changing finitely many initial values does not change $\phi$-impredictability, so we may assume $\phi \leq \phi'$. Let $\psi \in \TT$ be given by $\phi'$-impredictability, so
\[ \forall \chi \in \PP: \exists^\infty p: \psi(p) \in A \iff \chi(p, A \upharpoonright \phi'(p))\!\downarrow \]
In particular, we can restrict to $\hat \chi \in \PP$ which given $(p, w)$ first cuts off all but at most the first $\phi(p) \leq \phi'(p)$ symbols of $w$, and then apply $\chi$. This restriction is
\[ \forall \hat \chi \in \PP: \exists^\infty p: \psi(p) \in A \iff \chi(p, A \upharpoonright \phi(p))\!\downarrow, \]
which is just the definition of $\phi$-impredictability.
\end{proof}


\begin{theorem}
\label{thm:ExistsApredictable}
For every $\phi \in \TT$, there exists a $\phi$-impredictable $\Sigma^0_1$ set.
\end{theorem}

\begin{proof}
We construct an increasing total computable function $\psi$, and $A$ will be contained in its image. We want to have
\[ \psi(p) \in A \iff \chi(p, A \upharpoonright \phi(p))\!\downarrow \]
infinitely many times for all $\chi$. The way we construct $A$ is we go through $n \in \N$ and set either $n \notin A$, or set up the rule $n \in A \iff {\xi_n {\downarrow}}$ for some Turing machine $\xi_n$ whose behavior we describe (informally). We refer to those $n \in \N$ already considered as \emph{determined}.

List all partial functions $\chi$ in an infinite-to-one way, i.e.\ we consider functions $\chi$ successively, so that each $\chi$ appears infinitely many times. When considering $\chi$, we want to determine new values in $A$ so as to make sure that
\[ \psi(p) \in A \iff \chi(p, A \upharpoonright \phi(p))\!\downarrow \]
for at least one new $p \in \N$. Suppose we have already determined the first $m$ values of $A$, i.e.\ the word $w \in \{0,1\}^m$ such that the characteristic sequence of $A$ will begin with the word $w$ is already determined, and that we have not determined whether $n \in A$ for any $n \geq m$. The idea is that while we do not know what the word $w$ actually is, there are only $2^m$ possible choices, and we simply try all of them to get the equivalence above to hold for one new $p$. To achieve this we will determine whether $n \in A$ for some interval of choices $n \in [m, M']$.

Enumerate the words of $\{0,1\}^m$ as $w_0, w_1, ..., w_{2^m-1}$. Now, let $p_0 < p_1 < ... < p_{2^m-1}$ be minimal such that $\phi(p_i) \geq m$ for all $i$. Let $M = \max_j \phi(p_j)$ Set $n \notin A$ for all $n \in [m, M]$. As $\psi(p_i)$ for $i \in [0, 2^m-1]$ pick any distinct values greater than $M$, and determine $\psi(p_i) \in A \iff \chi(p_i, w_i \cdot 0^{\phi(p_i)-m})\!\downarrow$. If $w = w_i$, then for $p = p_i$ we have
\[ \psi(p) \in A \iff \chi(p_i, w_i \cdot 0^{\phi(p_i)-m}){\downarrow} \iff \chi(p, A \upharpoonright \phi(p)){\downarrow}, \]
because the characteristic sequence of $A$ indeed begins with the word $w_i \cdot 0^{\phi(p_i)-m}$. Thus, we have obtained a new value $p = p_i$ at which the statement is satisfied for $\chi$.

For all $n$ that are yet undetermined, but some larger number is determined, we determine $n \notin A$, so that we determine the values in some new interval $[m, M']$. We can then inductively continue to the next value of $\chi$. This process determines all values of $A$, and by construction $A$ is a recursively enumerable set which is $\phi$-impredictable.
\end{proof}

\begin{lemma}
\label{lem:Interreduction}
Let $A, B \subset \N$ and suppose $A$ many-one reduces to $B$, and $B$ wtt-reduces to $A$ with rate $\beta$. If $A$ is $\phi$-impredictable, then $B$ is $(\beta \circ \phi)$-impredictable.
\end{lemma}

\begin{proof}
Let $\phi' = \beta \circ \phi$. Let $f : \N \to \N$ be the many-one reduction from $A$ to $B$, and $g : \{0,1\}^* \to \{0,1\}^*$ the wtt-reduction from $B$ to $A$ with rate $\beta$. Since $A$ is $\phi$-impredictable there exists $\psi : \N \to \N$ such that
\[ \forall \chi \in \PP: \exists^\infty p: \psi(p) \in A \iff \chi(p, A \upharpoonright \phi(p)) {\downarrow}. \]
Setting $\psi' = f \circ \psi$, we have, by the definition of $f$, that
\[ \forall \chi \in \PP: \exists^\infty p: \psi'(p) \in B \iff \chi(p, A \upharpoonright \phi(p)) {\downarrow}. \]

For $\chi \in \PP$ define $\hat \chi \in \PP$ as follows: Given $(p, w)$, compute $g(w) = u$ and then evaluate $\chi(p, u \upharpoonright \phi'(p))$, so that
\[ \hat \chi(p, A \upharpoonright \phi(p)) {\downarrow} \iff \chi(p, g(A \upharpoonright \phi(p)) \upharpoonright \phi'(p)) {\downarrow}, \]
where we observe that $g(A \upharpoonright \phi(p)) \upharpoonright \phi'(p) = B \upharpoonright \phi'(p)$ (in particular $g$ indeed halts with an output so the formula makes sense).

Specializing the first quantifier, we have
\[ \forall \hat \chi \in \PP: \exists^\infty p: \psi'(p) \in B \iff \hat \chi(p, A \upharpoonright \phi(p)) {\downarrow}, \]
equivalently
\[ \forall \chi \in \PP: \exists^\infty p: \psi'(p) \in B \iff \chi(p, B \upharpoonright \phi'(p)), \]
so $\psi'$ proves that $B$ is $\phi'$-impredictable, as desired.
\end{proof}

We show that impredictability implies uncomputability.

\begin{proposition}
If $A$ is $\Pi^0_1$, then $A$ is not $\phi$-impredictable for any $\phi \in \TT$.
\end{proposition}

\begin{proof}
Suppose for a contradiction that $A$ is $\Pi^0_1$ and is $\phi$-impredictable with $\phi \in \TT$. Let $\psi \in \PP$ be such that
\[ \forall \chi \in \PP: \exists^\infty p: \psi(p) \in A \iff \chi(p, A \upharpoonright \phi(p))\!\downarrow. \]
In particular, this applies to the following $\chi$: given $(p, w)$, we ignore $w$ and if $\psi(p) \notin A$, then $\chi(p,w) \!\downarrow$, and otherwise $\chi(p,w)\!\uparrow$. This well-defines $\chi \in P$ since $A$ is $\Pi^0_1$, and for all $p \in \N$ we have
\[ \psi(p) \in A \iff \chi(p, A \upharpoonright \phi(p))\!\uparrow, \]
a contradiction.
\end{proof}

\section{Impredictable torsion groups}

\begin{definition}
We define a group $K(G,A,H)$ which depends on a choice of a finitely-generated groups $G, H$ (and choices of generating sets for them, kept implicit), and a set $A \subset \N$. First define a subshift on $G$ by
\[ X_A = \{x \in \{0,1\}^G \;|\; \sum_{g \in G} x_g \leq 2, \mbox{ and } \sum_{k \in \{g, h\}} x_k = 2 \implies d(g, h) \notin A \} \]
where $d : G \times G \to \N$ is the (left-invariant) word metric. To each $g \in G$ associate the bijection
\[ \hat g(x, h') = (g \cdot x, h'), \;\; \hat g : X_A \times H \to X_A \times H \]
i.e.\ the usual shift action in the first component, and to each $b \in \{0,1\}$ and $h \in H$ associate the bijection
\[ h_b(x, h') = (x, h^{1- |x_{e_G} - b|} \cdot h'), \;\; h_b : X_A \times H \to X_A \times H \]
where $h^{1 - |x_{e_G} - b|}$ evaluates to $e_H$ if $x_{e_G} = b$, and to $h$ otherwise. 
Define
\[ K(G,A,H) = \langle \{\hat g, h_b \;|\; g \in G, h \in H, b \in \{0,1\}\} \rangle \leq \Sym(X_A \times H) \]
\end{definition}

Obviously $K = K(G, A, H)$ is finitely-generated, and the implicit generators we use for $K$ are the ones from the definition, with $\hat g$ taken only for $g$ in the generating set of $G$, and $h_b$ only for $h$ in the generating set of $H$.

\begin{remark}
This group can be interpreted as a group of finite-state machines in the sense of \cite{BaKaSa16} (with obvious nonabelian generalization) when $H$ is finite, by simulating the actions of $\hat g$ by translations of the head, having $|H|$ states, and changing the state by the left-regular action of $H$ (if in the correct clopen set) when $h_b$ is applied. Again if $H$ is finite, $K$ can also be interpreted as a subgroup of the topological full group of the $G \times H$-subshift $X_A \times H$ under the action $(g, h) \cdot (x, h') = (gx, hh')$, by having $\hat g$ act by $(g, e_H)$ and having $h_B$ act by either $(e_G, h)$ or $(e_G, e_H)$ depending on $x_{e_G}$.
\end{remark}

We now prove some important technical properties of these groups, leading up to the proof that they give examples of bounded torsion groups with impredictable word problem.

\begin{lemma}
If $A' \subset A$, the defining action of $K(G,A,H)$ can be seen as a restriction of the action of $K(G,A',H)$ in a natural way, thus $K(G,A,H)$ is a quotient group of $K(G,A',H)$. 
\end{lemma}

\begin{proof}
Map generators to generators in the obvious way. The group $K(G, A', H)$ acts on $X_{A'} \times H \supset X_A \times H$ and the restriction of the action to $X_A \times H$ is precisely that of $K(G,A,H)$. Thus, all identities of $K(G,A',H)$ are identities also in $K(G,A,H)$.
\end{proof}

In particular, the defining action of $K(G,A,H)$ is always a restriction of $K(G,\emptyset,H)$.

\begin{lemma}
\label{lem:SplitEpi}
For any $G, A, H$, there exists a split epimorphism $K(G,A,H) \rightarrow G$.
\end{lemma}

\begin{proof}
We have $X_{\leq 1} = X_{\N} \subset X_A$, so the $\hat g$-translations act nontrivially on the $X_A$ component as the action on $X_{\leq 1}$ is the one-point compactification of the left-regular action of $G$ on itself. Observe also that the $X_A$-component is not modified by any of the maps $h_b$, so the action of $K_A$ on this component factors is just the shift action $G \curvearrowright X_A$. This gives a homomorphism $\gamma : K(G, A, H) \to G$, and the map $g \mapsto \hat g$ is a section for it.
\end{proof}

We refer to $\gamma$ as the \emph{natural epimorphism}.

\begin{lemma}
If $\WP(G)$ is decidable and $A$ is $\Sigma^0_1$, then $X_A$ is an effective subshift.
\end{lemma}

\begin{proof}
Since we can list elements of $A$, and can compute the distance between given group elements, we can forbid all finite patterns where two $1$s appear at $g, h$ with $d(g,h) \in A$.
\end{proof}

\begin{lemma}
\label{lem:RPandExpSpeed}
If $G$ has decidable word problem and $H$ is recursively presented, then the following statements hold.
\begin{itemize}
\item If $H$ has decidable word problem, then the word problem of $K(G, A, H)$ conjunctively reduces to $A$ with exponential rate.
\item If $A$ is $\Sigma^0_1$ then $K(G, A, H)$ is recursively presented.
\end{itemize}
\end{lemma}

\begin{proof}
Let $K = K(G, A, H)$ and let $S$ be the finite generating set of $K$. We begin with the proof of the latter item. We need to find a semi-algorithm that, given $w \in S^*$, halts if and only if $w$ represents the identity. For this, first consider the natural epimorphism image $\gamma(w) \in G$. Because $G$ is in particular recursively presented, we can first verify that $\gamma(w) = e_G$ (if not, then also $w \neq e_K$, and the computation diverges as desired).

Assuming $\gamma(w) = e_G$, we next check that $w$ acts trivially on all elements of $X_A \times H$. We define an action of the free group on generators $S$ 
on pairs $(P,h) \in \lang(X_\emptyset) \times H$ by
\[ \hat g(P, h') = (g \cdot P, h'), \]
(where $\dom(gP) = g \dom (P)$ and $gP_k = P_{g^{-1}k}$), and for $b \in \{0,1\}$ and $h \in H$ we map
\[ h_b(P, h') = (P, h^{1- |P_{e_G} - b|} h'), \]
when $e_G \in \dom (P)$, and $h_b(P, h') = (P, h')$ otherwise. 

If $|w| \leq n$ and $\gamma(w) = e_G$, clearly $w \approx e_K$ if and only if the action of $w$ fixes all $P \times h$ where $P \in \mathcal{L}_n(X_A), h \in H$. Naturally, if $\mathcal{P} \supset \lang_n(X_A)$ and the action fixes $(P, h)$ for all $P \in \mathcal{P}, h \in H$, then a fortiori $w \approx e_K$. We can verify this for a particular $(P, h)$ by using the fact $H$ is recursively presented. By the previous lemma, $X_A$ is effective, so we can enumerate upper approximations to $\lang_n(X_A)$ which eventually converge. In other words, we eventually obtain the set $\lang_n(X_A)$, and it $w \approx e_K$, then at this point (at the latest) we can conclude that $w$ indeed acts trivially and halts.

The proof of the first item is similar. To see that there is a wtt-reduction with exponential rate, observe that if we know the first $n$ values of $A$, then we can determine the legal contents of all $G$-patterns with domain $B_n(G)$ in $X_A$, and using this, and the decidable word problem of $H$, we can determine whether $w \approx e_K$ for any word with $|w| \leq n$. Since we list elements of groups in lexicographic order, the resulting rate $\beta$ is exponential, as there are exponentially many words $w$ with $|w| \leq n$.

To see that this is a conjunctive reduction, we observe that the reduction function $g : \{0,1\}^* \to \{0,1\}^*$ (computing initial segments of the word problem from initial segments of $A$) of the previous paragraph can be written uniformly for all sets $A$ so it is total computable, and that we should set $g(u)_i = 1$ if and only if the $i$th group element acts trivially on the set of patterns not containing elements of $A$. It can be checked by a terminating computation whether the group element acts nontrivially on some pattern containing only one $1$. Our query on $A$ should check that whenever the element acts nontrivially on a pattern with two $1$s, then $A$ contains the distance between the $1$s of the pattern. This corresponds to checking $u \geq w_{|u|,i}$ where $w_{|u|,i}$ lists these finitely many bad distances.
\end{proof}

\begin{lemma}
\label{lem:ManyOne}
Suppose $G$ has decidable word problem and $H$ is not abelian. Then $A$ many-one reduces to the word problem of $K(G, A, H)$.
\end{lemma}

\begin{proof}
For $n \geq 1$, let $g_n$ be any effective list of elements of $G$ satisfying $|g_n| = n$ (using the fact $G$ has decidable word problem), and consider $g_n' = [h'_1, h_1^{g_n}]$ where $[h', h] \neq e_H$ (using that $H$ is not abelian) and the action of $h_1^{g_n} = g_n h_1 g_n^{-1}$ is
\[ h_1^{g_n}(x, h') = (x, h^{1- |x_{g_n} - b|} \cdot h') \]
We have $g'_n \approx e_K$ if and only if $n \in A$. Namely, if $n \notin A$ then $g'_n$ acts nontrivially on
$(x, e_H)$ where $x \in X_A$ is the unique configuration satisfying $x_{e_G} = x_{g_n} = 1$, while if $n \in A$ then for all $x \in X_A$ either $x_{e_G} = 0$ or $x_{g_n} = 0$, and in either case a direct computation shows $g_n'(x, h) = (x, h)$ for all $h \in H$.
\end{proof}

\begin{lemma}
\label{lem:Torsion}
For any torsion groups $G, H$, we have $T_{K(G,A,H)} = O(T_G T_H)$.
\end{lemma}

\begin{proof}
Let $K = K(G,A,H)$. If $w \in K$ with $|w| \leq n$, then $\gamma(w) \in G$ with $|\gamma(w)| = m \leq n$. Let $k = T_G(m)$ (note that $k \;|\; T_G(n)$), so $\gamma(w^k) =  \gamma(w)^k = e_G$. If $w^k$ has order $\ell$ for all $w$, then $w^{\ell k} = (w^k)^\ell = e_K$, and thus $T_K(n) \leq \ell T_G(m)$, and we have shown $T_K = O(T_G T_H)$ as claimed.

So suppose that $\gamma(w) = e_G$, and consider the action on $(x,h) \in X_A \times H$. The action of $w$ shifts $x$ around, and based on its contents multiplies $h$ from the right by elements of $H$. For any fixed $x \in X_A$, $\gamma(w) = e_G$ implies that there exists $h_x \in H$ with $|h_x| \leq |w|$, such that $w \cdot (x, h) = (x,h_x h)$ for all $h \in H$. Since $h_x$ has order at most $T_H(|h'|)$, and $T_H(|h'|) \;|\; T_H(|w|)$, we have $w^{T_H(|w|)} \cdot (x, h) = (x, h_x^{T_H(|w|)} h) = (x, h)$, concluding the proof.
\end{proof}

\begin{theorem}
\label{thm:Construction}
Let $A \subset \N$ be $\Sigma^0_1$. For any f.g.\ torsion group $G$ with decidable word problem there exists a recursively presented torsion group $K$ with $T_K = \Theta(T_G)$, such that the word problem of $K$ conjunctively reduces to $A$ with exponential rate and $A$ many-one reduces to the word problem of $K$.
\end{theorem}

\begin{proof}
By the previous lemmas, if $K = K(G,A,H)$ for finitely-generated groups $G$ and $H$ which have decidable word problems, and $H$ is nonabelian with bounded torsion, then
\begin{itemize}
\item $K$ is recursively presented (Lemma~\ref{lem:RPandExpSpeed}),
\item $T_K \geq T_G$ because $G \leq K$ and by the choice of generators (Lemma~\ref{lem:SplitEpi}),
\item $T_K \leq T_G T_H = O(T_G)$ (Lemma~\ref{lem:Torsion}),
\item $\WP(K)$ conjunctively reduces to $A$ with exponential rate (Lemma~\ref{lem:RPandExpSpeed}),
\item $A$ many-one reduces to $\WP(K)$ (Lemma~\ref{lem:ManyOne}).
\end{itemize}
\end{proof}

In the previous theorem, the implicit constants for $\Theta$ can be taken to be~$1$ (for the lower bound) and~$6$ (for the upper bound, by setting $H = S_3$). Of course, if $G$ has bounded torsion, so does $K$.

\begin{theorem}
\label{thm:ImpredictableGroups}
Let $\phi$ be a total recursive function. Then there exists a recursively presented torsion group with bounded torsion, whose word problem is $\phi$-impredictable.
\end{theorem}

\begin{proof}
Let $A$ be a $\phi$-impredictable $\Sigma^0_1$ set (Theorem~\ref{thm:ExistsApredictable}) and apply the previous theorem to obtain a recursively presented torsion group $K$ with bounded torsion, such that the word problem of $K$ conjunctively reduces to $A$ with exponential rate $\beta$ and $A$ many-one reduces to the word problem of $K$. Then in particular the word problem of $K$ wtt-reduces to $A$ with exponential rate $\beta$. By Lemma~\ref{lem:Interreduction}, the word problem of $K$ is $(\beta \circ \phi)$-impredictable. In particular, it is $\phi$-impredictable by Lemma~\ref{lem:Faster}.
\end{proof}

\section{Application: Four heads are better than three}
\label{sec:SubshiftCorollary}


We first define group-walking automata and the subshifts they recognize. By $\pi_i$ we mean the projection to the $i$th coordinate of a finite Cartesian product. For $Q_i \not\ni 0$ a finite set write $X^1_{Q_i}$ for the subshift on a group $G$ clear from context containing those $x \in (Q_i \cup \{0\})^G$ satisfying $|\{g \in G \;|\; x_g \neq 0\}| \leq 1$.

\begin{definition}
\label{def:Automata}
Let $\Sigma$ be a finite alphabet. A \emph{$k$-headed group-walking automaton} on the full shift $\Sigma^G$ is a tuple $\mathcal{A} = (\prod_{i=1}^k Q_i, f, I, F, S)$, where $Q_1, Q_2, \ldots, Q_k$ are state sets not containing the symbol $0$, $I$ and $F$ are finite clopen subsets of the product subshift $Y = \prod_{i=1}^k X^1_{Q_i}$, and $f : \Sigma^G \times Y \to \Sigma^G \times Y$ is a cellular automaton satisfying $\pi_1 \circ f = \pi_1$ and
\[ \pi_i(\pi_2(f(x, y))) = 0^G \iff \pi_i(y) = 0^G \]
for all $x \in \Sigma^G$, $y \in Y$ and $i \in \{1, \ldots, k\}$.

For a $k$-headed automaton $\mathcal{A}$ as above, we denote by $\AU(\mathcal{A}) \subset \Sigma^G$ the subshift
\[ \{ x \in \Sigma^G \;|\; \forall g, h \in G, y \in I, n \in \N: h \cdot \pi_2(f^n(g \cdot x, y)) \notin F \}. \]
For $k \geq 1$, we denote by $\AU(G,k)$ the class of all subshifts $\AU(\mathcal{A})$ for $k$-headed automata $\mathcal{A}$, and $\AU(G,0)$ is the class of all $G$-SFTs. We also write $\AU(G) = \bigcup_{k \in \N} \AU(G, k)$. 
\end{definition}

This definition may seem cryptic on a first reading. Its details are unraveled in \cite{SaTo17} (see \cite{SaTo15} for a discussion of possible variants). Our interpretation of $X \in \AU(G, k)$ is that a $k$-headed group-walking automaton can define $X$, for a particular (in our opinion natural) way of defining subshifts by such automata.

Key points are that interpreting the $i$th track of $Y$ as giving the position or a head and its current state, $f$ is a local rule that tells how the heads move on configurations, and the assumptions imply that all heads are always present, are initialized in (roughly) the same position (described by $I$), have to join together to reject a configuration (described by $F$), and cannot communicate over distances (because $f$ is a cellular automaton).

The following observation is essentially Proposition~2 in \cite{SaTo17}, though here we ``complement'' the separating subshift, because the result of \cite{SaTo17} cannot be used with recursively presented groups.

\begin{lemma}
\label{lem:Four}
Let $G$ be a finitely-generated torsion group, let $T(n) = \max(n, T_G(n))$, and suppose that there exists a superexponential function $\zeta : \N \to \N$ such that the word problem of $G$ is $(\zeta \circ \zeta \circ T \circ \zeta \circ T \circ \zeta)$-impredictable. Then
\[ \AU(G \times \Z, 3) \subsetneq \AU(G \times \Z, 4). \]
\end{lemma}

We sketch the proof from \cite{SaTo17}, for our complemented definitions.

\begin{proof}[Proof sketch]
Let $\zeta$ be superexponential and recursive,  let $G$ be such that the word problem of $G$ is $\phi$-impredictable for $\phi = \zeta \circ \zeta \circ T \circ \zeta \circ T \circ \zeta$, and let $\psi$ be the corresponding function, so
\[ \forall \chi \in \PP: \exists^\infty p: \psi(p) \sim e_G \iff \chi(p, G \upharpoonright \phi(p))\!\downarrow. \]

For each $p \in \N$, let $x^p \in \{0,1\}^{G \times \Z}$ be the configuration where $x^p_{(g,n)} = 1$ if and only if $n \equiv 0 \bmod p$. Define $B = \{ p \in \N \;|\; \psi(p) \sim e_G \}$, and let $Y_B \subset \{0, 1\}^{G \times \Z}$ be the smallest subshift containing the configurations $x^p$ with $p \in B$, i.e.\ the forbidden patterns are (encodings of) the complement of the word problem of $G$ (together with the recursive set of patterns ensuring $Y_B \subset Y_\emptyset$). We clearly have $x^p \in Y_B \iff p \in B \iff \psi(p) \sim e_G$.

The crucial observation in \cite{SaTo17} was that for any fixed three-headed automaton $\mathcal{A}$, the function $(\zeta \circ T \circ \zeta \circ T \circ \zeta)(p)$ eventually bounds how far the $G$-projections of the heads can be from each other during valid runs on a configuration $x^p$, which means that we can construct a Turing machine that, given access to an oracle for the initial segment $G \upharpoonright \phi(p)$ of the word problem (we add an extra $\zeta$ since we linearize the word problem), we can simulate all runs of $\mathcal{A}$ on the configurations $x^p$ (observe that there are essentially only $p$ different starting positions that need to be considered), halting if and only if one of them halts (i.e.\ the finite clopen set $F$ is entered and the configuration is forbidden).

Suppose for a contradiction that $\mathcal{A}$ is a three-headed automaton that defines $Y_B$. Letting $\chi$ be the Turing machine described above which simulates $\mathcal{A}$, we have for all large enough $p$ that
\[ {\chi(p, G \upharpoonright \phi(p)) {\downarrow}} \iff x^p \notin Y_B \] 
but by the definition of $\psi$ there exist arbitrarily large $p$ such that
\[ \psi(p) \sim e_G \iff \chi(p, G \upharpoonright \phi(p))\!\downarrow \iff x^p \notin Y_B \iff \psi(p) \not\sim e_G, \]
a contradiction.

On the other hand $Y_B \in \AU(G \times \Z, 4)$ since it is intrinsically $\Pi^0_1$ in the sense of \cite{SaTo17}, by a similar proof as in \cite{SaTo17}, observing that using an oracle for the word problem of $G$ we can easily forbid all $x^p$ with $p \notin B$.
\end{proof}

Four heads are now seen to be better than three:

\begin{proof}[Proof of Theorem~\ref{thm:FourBetterThanThree}]
By Theorem~\ref{thm:ImpredictableGroups}, there exists a group $G$ which has bounded torsion and has $\exp \circ \exp \circ \exp \circ \exp \circ \exp$-impredictable word problem. Such a group is $(\zeta \circ \zeta \circ T \circ \zeta \circ T \circ \zeta)$-impredictable for some superexponential function $\zeta$, where $T(n) = n$, and thus the previous lemma implies $\mathcal{S}(G \times \Z, 3) \subsetneq \mathcal{S}(G \times \Z, 4)$.
\end{proof}

\begin{remark}
For maximal ``automaticity'',\footnote{Using the Grigorchuk group or another torsion automata group, the statement of our main result is somewhat amusing, in that it states that a group of finite-state \textbf{automata} acting on a subshift on an \textbf{automata} group admits subshifts definable by group-walking \textbf{automata} with four heads but not three.} 
or to avoid using f.g.\ bounded torsion groups (whose infiniteness is rather difficult to verify for mortals), one can replace the use of bounded torsion groups by any torsion group with recursive torsion function. For example using the Grigorchuk group as $G$ in the construction of $K$ (then $T_K(n) \leq O(n^3)$ by \cite[Theorem~VIII.70]{Ha00}), and using Theorem~\ref{thm:Construction} directly, the previous proof goes through with the same exponential tower. Automata groups in the sense of \cite{BaSi10} cannot have bounded torsion by Zelmanov's theorem \cite{Ha00,Ze91,Ze91a}.
\end{remark}

\bibliographystyle{plain}
\bibliography{../../../../bib/bib}{}

\end{document}